\documentclass[10pt,a4paper]{article}
\usepackage{fullpage}
\usepackage{amsfonts,amsmath,amssymb}
\usepackage{amsthm}
\usepackage{graphicx}\usepackage{sectsty}
\usepackage{authblk}

\theoremstyle{plain}
\newtheorem{theorem}{Theorem}[section]
\newtheorem{lemma}[theorem]{Lemma}

\theoremstyle{definition}
\newtheorem{definition}[theorem]{Definition}
\theoremstyle{remark}

\numberwithin{equation}{section}

\sectionfont{\large}
\newenvironment{acknowledgement}[1][Acknowledgement
]{\begin{trivlist} \item[\hskip \labelsep {\bfseries
#1}]}{\end{trivlist}}

\begin{document}
\title{Partial Information for Inverse Spectral Uniqueness \\in Vibration System with Multiple Frozen Arguments }
\author{\small Lung-Hui Chen}
\affil{\footnotesize General Education Center, Ming Chi University of Technology, New Taipei City, 24301, Taiwan;\\
Email: mr.lunghuichen@gmail.com.}
\date{}
\maketitle
\begin{abstract}
In this paper, we investigate the inverse spectral problem of the Sturm-Liouville operator 
with many frozen arguments fixed at the points $\{a_{1}, a_{2},\ldots,a_{N}\}$ in $(0,\pi)$. We start with counting the zeros or the eigenvalues of characteristic function, and then discuss how certain information provided  a priori on the point set $\{a_{1}, a_{2},\ldots,a_{N}\}$  would affect  the uniqueness or non-uniqueness of this vibration system with many frozen points. The knowledge at the frozen or regulator points are practical in many on-site problems. Parallelly, certain irrational independence assumption assures the inverse spectral uniqueness as well.
\end{abstract}
\section{Introduction}
In this paper, the author studies the inverse spectral problem of recovering the real-valued potential function $q(x)\in L_{\mathbb{R}}^{2}(0,\pi)$ from the spectrum of the boundary value problem
$$L=L(q(x),a_{1}, a_{2},\ldots,a_{N},\alpha,\beta)$$ 
of the form
\begin{eqnarray}
\left\{
\begin{array}{ll}\label{1.1}

ly:=-y''(x)+q(x)\sum_{i=1}^{N} y(a_{i})=\lambda y(x),\,0<x<\pi;\vspace{12pt}\\
y^{(\alpha)}(0)=y^{(\beta)}(\pi)=0,\label{1.22}
\end{array}
\right.
\end{eqnarray}
in which $\lambda=\rho^{2},\,\rho\in\mathbb{C},$ is the spectral parameter, and $\alpha$, $\beta\in\{0,1\}$. Here we assume that $0<a_{1}<a_{2}<\cdots<a_{N}<\pi$. The operator $l$ is called the Sturm-Liouville type of operator with $N$ frozen arguments with boundary condition of model~(\ref{1.1}).
\par
In the literature, the equation~(\ref{1.1}) is one kind of loaded differential equations. The equations model many experiments in the fields of applied physics, mathematical physics, and many engineering works focused on the control and measurement of certain vibrations, magnetic shielding, or certain oscillatory systems \cite{Dikinov,Iskenderov,Lomov,Nakhushev}. In many oscillatory systems, one needs to place certain measurement instruments or sensors to monitor the physical states empirically that makes the frozen points in the model~(\ref{1.1}). 
\par
The inverse spectral problems for Sturm-Liouville operators with frozen arguments were previously studied in \cite{Alb,Bon,BK,Shieh} to name a few, and highly pursued by many interested mathematicians and engineers. Recently, the author has published a preprint at arXiv \cite{Chen}. Among the rapidly growing researches as aforementioned, there are a few discussions around the rational or irrational independency of the frozen points $\{a_{1}/\pi, a_{2}/\pi,\ldots,a_{N}/\pi\}$ which is pursued by many mathematicians in the last decade. In this paper, we would also like to show that the rational  or irrational perity of $\{a_{1}/\pi, a_{2}/\pi,\ldots,a_{N}/\pi\}$ both play a significant role while certain knowledge on $q(x)$ is provided a priori. There are many treatments to provide such knowledge through the balance of zero density on both sides of characteristic equations.

\section{Preliminaries and Lemmata}
We begin by referencing several identities from \cite{Shieh}, specifically their identities (25), (26), (27), and (28). Here are the characteristic functions all we need for the system~(\ref{1.1}):
\begin{eqnarray}\hspace{-6pt}
\Delta_N^{(0,0)}(\lambda) = \begin{vmatrix}
\frac{\sin \rho a_1}{\rho} & \hspace{5pt}\int_0^{a_1} q(t) \frac{\sin \rho (a_1 - t)}{\rho} dt-1 & 1 & 1 & 1 & \dots & 1 & 1 \\
\frac{\sin \rho a_2}{\rho} & \int_0^{a_2} q(t) \frac{\sin \rho (a_2 - t)}{\rho} dt & -1 & 0 & 0 & \dots & 0 & 0 \\
\frac{\sin \rho a_3}{\rho} & \int_0^{a_3} q(t) \frac{\sin \rho (a_3 - t)}{\rho} dt & 0 & -1 & 0 & \dots & 0 & 0 \\
\vdots & \vdots & \vdots & \vdots & \vdots & \ddots & \vdots & \vdots \\
\frac{\sin \rho a_N}{\rho} & \int_0^{a_N} q(t) \frac{\sin \rho (a_N - t)}{\rho} dt & 0 & 0 & 0 & \dots & 0 & -1 \\
\frac{\sin \rho \pi}{\rho} & \int_0^{\pi} q(t) \frac{\sin \rho (\pi - t)}{\rho} dt & 0 & 0 & 0 & \dots & 0 & 0
\end{vmatrix},
\end{eqnarray}
\begin{eqnarray}\hspace{-5pt}
\Delta_N^{(0,1)}(\lambda) = \begin{vmatrix}
\frac{\sin \rho a_1}{\rho} & \hspace{4pt}\int_0^{a_1} q(t) \frac{\sin \rho (a_1 - t)}{\rho} dt - 1 & 1 & 1 & 1 & \dots & 1 & 1 \\
\frac{\sin \rho a_2}{\rho} & \int_0^{a_2} q(t) \frac{\sin \rho (a_2 - t)}{\rho} dt & -1 & 0 & 0 & \dots & 0 & 0 \\
\frac{\sin \rho a_3}{\rho} & \int_0^{a_3} q(t) \frac{\sin \rho (a_3 - t)}{\rho} dt & 0 & -1 & 0 & \dots & 0 & 0 \\
\vdots & \vdots & \vdots & \vdots & \vdots & \ddots & \vdots & \vdots \\
\frac{\sin \rho a_N}{\rho} & \int_0^{a_N} q(t) \frac{\sin \rho (a_N - t)}{\rho} dt & 0 & 0 & 0 & \dots & 0 & -1 \\
\cos \rho \pi & \int_0^{\pi} q(t)\cos \rho (\pi - t)dt & 0 & 0 & 0 & \dots & 0 & 0
\end{vmatrix},
\end{eqnarray}
\begin{eqnarray}\hspace{-4pt}
\Delta_N^{(1,0)}(\lambda) = \begin{vmatrix}
\cos \rho a_1& \hspace{-1pt}\int_0^{a_1} q(t) \frac{\sin \rho (a_1 - t)}{\rho} dt - 1 & 1 & 1 & 1 & \dots & 1 & 1 \\
\cos \rho a_2& \int_0^{a_2} q(t) \frac{\sin \rho (a_2 - t)}{\rho} dt & -1 & 0 & 0 & \dots & 0 & 0 \\
\cos \rho a_3 & \int_0^{a_3} q(t) \frac{\sin \rho (a_3 - t)}{\rho} dt & 0 & -1 & 0 & \dots & 0 & 0 \\
\vdots & \vdots & \vdots & \vdots & \vdots & \ddots & \vdots & \vdots \\
\cos \rho a_N& \int_0^{a_N} q(t) \frac{\sin \rho (a_N - t)}{\rho} dt & 0 & 0 & 0 & \dots & 0 & -1 \\
\cos \rho \pi& \int_0^{\pi} q(t) \frac{\sin \rho (\pi - t)}{\rho} dt & 0 & 0 & 0 & \dots & 0 & 0
\end{vmatrix},\label{224}
\end{eqnarray}
and
\begin{eqnarray}\hspace{-4pt}
\Delta_N^{(1,1)}(\lambda) = \begin{vmatrix}
\cos \rho a_1 &\hspace{-7pt} \int_0^{a_1} q(t) \frac{\sin \rho (a_1 - t)}{\rho} dt - 1 & 1 & 1 & 1 & \dots & \hspace{-2pt}1 & 1 \\
\cos \rho a_2 & \int_0^{a_2} q(t) \frac{\sin \rho (a_2 - t)}{\rho} dt & -1 & 0 & 0 & \dots & 0 & 0 \\
\cos \rho a_3 & \int_0^{a_3} q(t) \frac{\sin \rho (a_3 - t)}{\rho} dt & 0 & -1 & 0 & \dots & 0 & 0 \\
\vdots & \vdots & \vdots & \vdots & \vdots & \ddots & \vdots & \vdots \\
\cos \rho a_N & \int_0^{a_N} q(t) \frac{\sin \rho (a_N - t)}{\rho} dt & 0 & 0 & 0 & \dots & 0 & -1 \\
-\rho\sin\rho\pi & \int_0^{\pi} q(t)  \cos \rho (\pi - t) dt & 0 & 0 & 0 & \dots & 0 & 0
\end{vmatrix}.
\end{eqnarray}
For any two $q^{1}(x)$, $q^{2}(x)$ in $L^{2}_{\mathbb{R}}(0,\pi)$ that sharing the same characteristic function $\Delta^{(\alpha,\beta)}(\lambda)$, $\alpha,\,\beta\in\{0,1\}$, we define
\begin{equation}
\hat{q}(x):=q^{1}(x)-q^{2}(x).
\end{equation}
Then, we deduce from~(\ref{224})
the following integral identity
\begin{equation}\label{2.2}
\int_{0}^{\pi}\sin{\rho(\pi-t)}\hat{q}(t)dt\Big\{\sum_{i=1}^{N}\cos{\rho a_{i}}\Big\}
=\cos{\rho \pi}\Big\{\sum_{i=1}^{N}\int_{0}^{a_{i}}\sin{\rho(a_{i}-t)}\hat{q}(t)dt\Big\},
\end{equation}
and the other cases have similar expressions. The main point of this paper is to count the roots of identity~(\ref{2.2}).
\par
For reader's convenience, we include the following classical lemma proved by E. C. Titchmarsh \cite[Theorem\,IV]{Titchmarsh} which deals with the zero set of a Fourier transform. A modern version for functions in distributional sense can be found in \cite[Lemma\,1.3]{Tang}.
\begin{lemma}[Titchmarsh]\label{L2.1}
Let $u\in\mathcal{E}'(\mathbb{R})$, the space of distributions with compact support, then
\begin{equation}N_{\mathcal{F}(u)}(r)= \frac{|\mbox{c.h. supp } u|}{\pi}(r + o(1)),\label{T1}
\end{equation} in which
\begin{equation}
N_{f}(r)=\sum_{|z|\leq r}\frac{1}{2\pi}\oint_{z}\frac{f'(\omega)}{f(\omega)}d\omega,\,z\in\mathbb{C},\label{T2}
\end{equation}
and the phrase $|\mbox{c.h. supp }u|$ means  the convex hull of the effective support  of $u$. Moreover, $N_{f}(r)$ is the counting function of the zeros of $f$ inside a ball of radius $r$ in $\mathbb{C}$, and we count the zeros according to their multiplicities.
\end{lemma}
\begin{proof}
We refer all details to \cite{Levin,Tang,Titchmarsh}.
\end{proof}
\begin{definition}
We denote the following quantity as the zero density of an entire function $f$ of finite type:
\begin{equation}
\delta(f):=\lim_{r\rightarrow\infty}\frac{N_{f}(r)}{r}.
\end{equation}
\end{definition}
\begin{lemma}\label{L2.3}
Let $f$, $g$ be two entire functions of finite type with zero density functions $\delta(f)$ and
$\delta(g)$ respectively. Then, 
\begin{eqnarray*}
&&\delta(fg)=\delta(f)+\delta(g); \\
&&\delta(f+g)=\max\{\delta(f),\delta(g)\},\end{eqnarray*}
if the densities of two functions are not equal.
\end{lemma}
\begin{proof}
We refer to B. Levin's book \cite[p.\,52]{Levin} for more detailed discussion. 

\end{proof}
\begin{lemma}\label{L2.4}
For $f(x)$ in $L^{2}(0,\pi)$, $\rho\in\mathbb{C}$, and $\int_{0}^{\pi}e^{-i\rho x}f(x)dx=\int_{0}^{\pi}\cos{\rho x}f(x)dx-i\int_{0}^{\pi}\sin{\rho x}f(x)dx$, the entire functions $$\int_{0}^{\pi}e^{-i\rho x}f(x)dx,\, \int_{0}^{\pi}\cos{\rho x}f(x)dx,\mbox{ and }\int_{0}^{\pi}\sin{\rho x}f(x)dx$$ have the same zero density in $\mathbb{C}$.
\end{lemma}
\begin{proof}
Let us simply apply Lemma \ref{L2.3}.
\end{proof}
The point is to see that the zero densities of $$\int_{0}^{\pi}e^{-i\rho x}f(x)dx,\, \int_{0}^{\pi}\cos{\rho x}f(x)dx,\mbox{ and }\int_{0}^{\pi}\sin{\rho x}f(x)dx$$
is all about the effective integral supports of these functions \cite{Levin,Tang,Titchmarsh}.

\section{Results}
In this section, we will discuss that the inverse spectral uniqueness and non-uniqueness of potential function $q(x)$ varies over certain partial information provided a priori  on $q(x)$. The rational or irrational independence of $\{a_{1}/\pi,a_{2}/\pi,\ldots,a_{N}/\pi\}$ may or may not play a role. The part 2 in the following theorem is already proved in \cite{Shieh}. 
\begin{theorem}[Partial Information]
For any two potential functions $q^{1}(x)$, $q^{2}(x)$ in $L^{2}_{\mathbb{R}}(0,\pi)$ sharing the identical characteristic function $\Delta^{(\alpha,\beta)}(\lambda)$, where $\alpha,\,\beta\in\{0,1\},$ and $$\hat{q}:=q^{1}(x)-q^{2}(x),$$ we state the following results:\begin{enumerate}
\item If $\delta\Big\{\sum_{i=1}^{N}\cos{\tilde{\rho} a_{i}}\Big\}>0$, where $\tilde{\rho}\in\{\frac{2n+1}{2}\}_{n\in\mathbb{Z}}$, then $\hat{q}\not\equiv0$ in $L^{2}_{\mathbb{R}}(0,\pi)$;
\item If for any $\hat{q}(x)\in L^{2}_{\mathbb{R}}(0,\pi)$, and $\sum_{i=1}^{N}\cos{\rho a_{i}}$ assumed not vanishing for any $\rho$ in $\mathbb{C}$, then $\hat{q}(x)\equiv 0$;
\item If $\hat{q}$ is non-vanishing in some suitable neighborhoods of $0$ and of $\pi$, but vanishing in some neighborhoods of $a_{N}$, then $\hat{q}(x)\not\equiv 0$ strictly inside $L^{2}_{\mathbb{R}}(0,\pi)$ except near $a_{N}$.
\end{enumerate}
\end{theorem}
\begin{proof}
To prove part 1, we begin with equation~(\ref{2.2}):
\begin{equation}\label{3.1}
\int_{0}^{\pi}\sin{\rho(\pi-t)}\hat{q}(t)dt\Big\{\sum_{i=1}^{N}\cos{\rho a_{i}}\Big\}
=\cos{\rho \pi}\Big\{\sum_{i=1}^{N}\int_{0}^{a_{i}}\sin{\rho(a_{i}-t)}\hat{q}(t)dt\Big\}.
\end{equation}
Let $\mathcal{Z}$ be the zero set of $\cos{\rho \pi}$, that is,
\begin{equation}
\mathcal{Z}=\{\frac{2n+1}{2}\}_{n\in\mathbb{Z}},
\end{equation}
and its zero density is $\delta(\mathcal{Z})=1$.
Plugging the set $\mathcal{Z}$ into~(\ref{3.1}), we obtain that 
\begin{equation}
\int_{0}^{\pi}\sin{\tilde{\rho}(\pi-t)}\hat{q}(t)dt\Big\{\sum_{i=1}^{N}\cos{\tilde{\rho} a_{i}}\Big\}
=0,\mbox{ for }\tilde{\rho}\in\mathcal{Z}.
\end{equation}
Using the theorem assumption, the zero density $$0<\delta\Big\{\sum_{i=1}^{N}\cos{\tilde{\rho} a_{i}}\Big\}<1.$$ 
Hence, there exists some $\rho_{0}=\frac{2n_{0}+1}{2}\neq0$ in $\mathcal{Z}$, and $n_{0}\geq1$ such that
\begin{equation}
\int_{0}^{\pi}\sin{\rho_{0}(\pi-t)}\hat{q}(t)dt\neq0.
\end{equation}
Using the transformation $\frac{2n_{0}+1}{2}(\pi-t)=n_{0}s$, we deduce that
\begin{equation}\label{FC}
\int_{0}^{\frac{2n_{0}+1}{2n_{0}}}\sin(n_{0}s)\,\hat{q}(\pi-\frac{2n_{0}}{2n_{0}+1}s)\frac{2n_{0}}{2n_{0}+1}ds\neq0,\,n_{0}\in\mathbb{Z}.
\end{equation}
Because $\{\sin{n x}\}_{n=1}^{\infty}$ is a basis for $L^{2}_{\mathbb{R}}(0,\pi)$ and we have non-zero $n_{0}$-th Fourier coefficient as in~(\ref{FC}), $\hat{q}(x)\not\equiv0$ in  $L^{2}_{\mathbb{R}}(0,\pi)$. There is no uniqueness in  $L^{2}_{\mathbb{R}}(0,\pi)$.

\par

To prove part 2, we suppose that $\hat{q}(x)\not\equiv0\in L^{2}(0,\pi)$, and note from~(\ref{3.1}) that
\begin{equation}\label{3.4}
\int_{0}^{\pi}\sin{\rho(\pi-t)}\hat{q}(t)dt\Big\{\sum_{i=1}^{N}\cos{\rho a_{i}}\Big\}
=\cos{\rho \pi}\Big\{\sum_{i=1}^{N}\int_{0}^{a_{i}}\sin{\rho(a_{i}-t)}\hat{q}(t)dt\Big\}.
\end{equation}
Now we count the zeros on both sides of~(\ref{3.4}) for $\rho\in\mathbb{C}$ again. Firstly, we consider the zero set on left of~(\ref{3.4}).
By the assumption of part 2, we have assumed $\Big\{\sum_{i=1}^{N}\cos{\rho a_{i}}\Big\}\neq0$ for any $\rho\in\mathbb{C}$. Furthermore,
\begin{equation}\label{ccc}
\delta(\int_{0}^{\pi}\sin{\rho(\pi-t)}\hat{q}(t)dt)\leq 1,
\end{equation}
by applying Lemma \ref{L2.1} and Lemma \ref{L2.4}. If there is any element on the right hand side of~(\ref{3.4}):
\begin{equation}\label{RR}
\mathcal{R}:=\Big\{\int_{0}^{a_{i}}\sin{\rho(a_{i}-t)}\hat{q}(t)dt\Big\}_{i=1}^{N}
\end{equation}
a non-trivial function with non-trivial support, then  we would obtain zero density to be strictly greater than $1$ on the right hand side of~(\ref{3.4}). On the other case, if all of the functions in $\mathcal{R}$ were trivial functions, the density function in~(\ref{ccc}) would not hold as an equality due to the smaller effective support. This is contradiction, and then $\hat{q}(x)\equiv0$ in $L^{2}(0,\pi)$.
\par
To prove part 3, we suppose that $\hat{q}(x)\equiv 0$ on some $(\delta,\pi-\delta)$ except near $a_{N}$, where $0<\delta<a_{1}$, $a_{N}<\pi-\delta$, and go back again to
\begin{equation}\label{3.9}
\delta\Big\{\int_{0}^{\pi}\sin{\rho(\pi-t)}\hat{q}(t)dt\Big\}+\delta\Big\{\sum_{i=1}^{N}\cos{\rho a_{i}}\Big\}
=\delta\Big\{\cos{\rho \pi}\Big\}+\delta\Big\{\sum_{i=1}^{N}\int_{0}^{a_{i}}\sin{\rho(a_{i}-t)}\hat{q}(t)dt\Big\}.
\end{equation}
By referencing Levin's book \cite[p.\,52]{Levin}, we have
\begin{equation}
\delta\Big\{\sum_{i=1}^{N}\int_{0}^{a_{i}}\sin{\rho(a_{i}-t)}\hat{q}(t)dt\Big\}\leq\max_{i=1}^{N}\Big\{\delta\{\int_{0}^{a_{i}}\sin{\rho(a_{i}-t)}\hat{q}(t)dt\}\Big\}<\frac{a_{N}}{\pi},
\end{equation}
in which the last equality does not hold due the smaller effective support of $\hat{q}(a_{N}-t)$ inside the integral.
From the assumption, we obtain that zero density on the left hand side of~(\ref{3.9}) is equal to $1+a_{N}/\pi$, and however, the zero density on the right hand side is strictly less than $1+a_{N}/\pi$.
The theorem is thus proven.

\end{proof}
\begin{acknowledgement}
The author appreciates National Science and Technology Council in Taiwan for the funding under the project number NSTC 113-2115-M-131-001. The author declares there is no conflict of interest in any sort. There is no data generated in this research.
\end{acknowledgement}

\end{document}